\documentclass[ 11pt,a4paper]{article}
\usepackage[all]{xy}
\usepackage[utf8]{inputenc}
\usepackage{amsmath,bm}
\usepackage{amssymb}
\usepackage{amsthm}
\usepackage{mathtools}
\usepackage{cases}
\usepackage{float}
\usepackage{here}
\usepackage{comment}
\usepackage{ulem}
\usepackage[T1]{fontenc}
\usepackage{textcomp}
\usepackage{graphicx, color}
\usepackage{tikz,epic,eso-pic}

\theoremstyle{plain}
\newtheorem{thm}{Theorem}[section]
\newtheorem{lem}[thm]{Lemma}
\newtheorem{cor}[thm]{Corollary}
\newtheorem{prop}[thm]{Proposition}

\theoremstyle{definition}
\newtheorem{df}[thm]{Definition}

\newtheorem{eg}[thm]{Example}

\theoremstyle{remark}
\newtheorem{rem}[thm]{Remark}

\numberwithin{equation}{section}
\allowdisplaybreaks[3]

\newcommand{\R}{\mathbb{R}}

\newcommand{\Z}{\mathbb{Z}}

\newcommand{\sSet}{\Set^{\Delta^{\rm op}}}

\newcommand{\Set}{{\sf Set}}

\newcommand{\tS}{{\tt S}}
\newcommand{\tT}{{\tt t}}
\newcommand{\tH}{{\tt h}}
\newcommand{\tV}{{\tt v}}
\newcommand{\tHV}{{\tt hv}}

\newcommand{\wt}{\widetilde}

\newcommand{\del}{\partial}

\newcommand{\too}{\longrightarrow}

\DeclareMathOperator{\MC}{\sf MC}
\DeclareMathOperator{\MH}{\sf MH}

\usepackage{xcolor}

\title{
Magnitude homology and homotopy type of metric fibrations
}

\author{Yasuhiko Asao, Yu Tajima and  Masahiko Yoshinaga}

\date{\empty}
\begin{document}

\maketitle
\begin{abstract}
    In this article, we show that each two metric fibrations with a common base and a common fiber have isomorphic magnitude homology, and even more, the same magnitude homotopy type. That can be considered as a generalization of a fact proved by T. Leinster that  the magnitude of a metric fibration with finitely many points is a product of those of the base and the fiber. We also show that the definition of  the magnitude homotopy type due to the second and the third authors is equivalent to the geometric realization of Hepworth and Willerton's pointed simplicial set.
\end{abstract}
\section{Introduction}
The notion of a {\it metric fibration} was defined by T. Leinster in his study of magnitude (\cite{L2}). It is a ``fibration in the category of metric spaces'',  defined analogously to the Grothendieck fibrations of small categories, where one sees a metric space as an category enriched over $([0, \infty), \geq, +)$. Based on the fact that a Grothendieck fibration can also be considered as a lax functor, the first author later provided an analogous description for the metric fibration (\cite{A2}). A remarkable property of the metric fibration is that the magnitude of the total space of a metric fibration is a product of those of the base and the fiber if they are finite metric spaces (\cite{L2} Theorem 2.3.11).  In this article, we show that the same is true for the magnitude homology and the magnitude homotopy type of a metric fibration possibly with infinitely many points. Namely we have the following.
\begin{thm}[Corollary \ref{main1}]
    Let $\pi : E \too B$ a metric fibration, and let $F$ be its fiber. For $\ell>0$, we have a homotopy equivalence 
    \[
    \MC^{\ell}_\ast(E) \simeq \bigoplus_{\ell_\tV + \ell_\tH = \ell} \MC^{\ell_\tV}_\ast(F)\otimes \MC^{\ell_\tH}_\ast(B),
    \]
    where $\MC$ denotes the magnitude chain complex.
\end{thm}
\begin{thm}[Corollary \ref{main2}, Corollary \ref{main3}]
    Let $\pi : E \too B$ be a metric fibration and let $F$ be its fiber. Then we have a homotopy equivalence 
    \[
    |{\sf M}^{\ell}_\bullet(E)| \simeq  \bigvee_{\ell_\tV + \ell_\tH = \ell}|{\sf M}^{\ell_\tV}_\bullet(F)| \wedge |{\sf M}^{\ell_\tH}_\bullet(B)|,
    \]
     where $|{\sf M}^{\ell}_\bullet(-)|$ is the geometric realization of the Hepworth and Willerton's pointed simplicial set (\cite{HW}).
\end{thm}
In particular, we give an another proof for  the K\"{u}nneth theorem for magnitude homology proved by Hepworth and Willerton (\cite{HW} Proposition 8.4). 

We use the terminology {\it magnitude homotopy type} as a CW complex whose singular homology is isomorphic to the magnitude homology of some metric sapce. Such a topological space first appeared in Hepworth and Willerton's paper (\cite{HW} Definition 8.1), and later the second and the third author gave another definition (\cite{TY}) by generalizing the construction for graphs due to the first author and Izumihara (\cite{AI}). In their paper, the second and the third author stated that the both definitions of the magnitude homotopy type, theirs and Hepworth-Willeton's, are equivalent without a proof. We gave a proof for it in the appendix (Proposition \ref{TYprof}).

The main idea of the proof of our main results is to construct a contractible subcomplex $D^\ell_\ast(E)$ of the magnitude chain complex $\MC^\ell_\ast(E)$ for a metric fibration $\pi : E \too B$. We have the following isomorphism (Proposition \ref{chainisom})

\[
    \MC^{\ell}_{\ast}(E)/D_\ast^{\ell}(E)\cong \bigoplus_{\ell_{\tV} + \ell_{\tH} = \ell} \MC^{\ell_{\tV}}_\ast(F)\otimes \MC^{\ell_{\tH}}_\ast(B),
    \]
where $F$ is the fiber of $\pi$. To find such a subcomplex $D^\ell_\ast(E)$, we use the  classification {\it horizontal, vertical, tilted}, of pairs of points of $E$ as in Figure \ref{figcla}. We define (Definition \ref{defd}) a submodule $D^\ell_n(E)$ of $\MC^\ell_n(E) \subset \Z E^{n+1}$ as one generated by tuples $(x_0, \dots, x_n)$ that contains tilted pair $(x_s, x_{s+1})$ earlier than horizontal-vertical triple $(x_t, x_{t+1}, x_{t+2})$ (namely $s+1 \leq t$), or contains  horizontal-vertical triple $(x_t, x_{t+1}, x_{t+2})$ earlier than tilted pair $(x_s, x_{s+1}) $ (namely $t+2 \leq s$). We show that $D^\ell_\ast(E)$ is a subcomplex of $\MC^\ell_\ast(E)$ (Lemma \ref{subchain}), and that it is contractible (Proposition \ref{dcontractible}) by using the algebraic Morse theory. For the magnitude homotopy type, we basically follow the same argument using {\it $\Delta$-sets} instead of chain complexes (Section \ref{htpytype}).

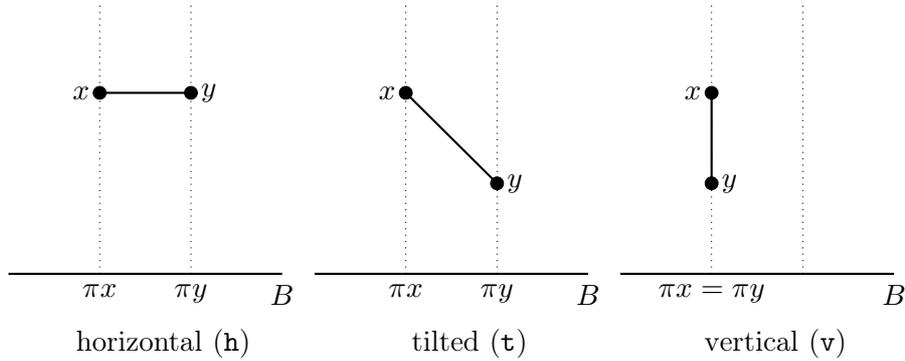
\begin{figure}[H]
\centering
\begin{tikzpicture}[scale=1.2]

\filldraw[fill=black, draw=black] (2, 2) circle (2pt) node[right] {$y$} ;
\filldraw[fill=black, draw=black] (1, 2) circle (2pt) node[left] {$x$} ;
\draw (1, 0)  node[below] {$\pi x$} ;
\draw (2, 0)  node[below] {$\pi y$} ;
\draw (1.7, -0.5)  node[below] {{\text horizontal ($\tH$)}} ;

\draw[thick] (0, 0)--(3, 0) node[below] {$B$};
\draw[thick] (1, 2)--(2, 2);
\draw[dotted] (1, 0)--(1, 3);
\draw[dotted] (2, 0)--(2, 3);

\end{tikzpicture}
\begin{tikzpicture}[scale=1.2]

\filldraw[fill=black, draw=black] (2, 1) circle (2pt) node[right] {$y$} ;
\filldraw[fill=black, draw=black] (1, 2) circle (2pt) node[left] {$x$} ;
\draw (1, 0)  node[below] {$\pi x$} ;
\draw (2, 0)  node[below] {$\pi y$} ;
\draw (1.7, -0.5)  node[below] {{\text tilted ($\tT$)}} ;

\draw[thick] (0, 0)--(3, 0) node[below] {$B$};
\draw[thick] (1, 2)--(2, 1);
\draw[dotted] (1, 0)--(1, 3);
\draw[dotted] (2, 0)--(2, 3);

\end{tikzpicture}
\centering
\begin{tikzpicture}[scale=1.2]

\filldraw[fill=black, draw=black] (1, 1) circle (2pt) node[right] {$y$} ;
\filldraw[fill=black, draw=black] (1, 2) circle (2pt) node[left] {$x$} ;
\draw (1, 0)  node[below] {$\pi x= \pi y$} ;
\draw (1.7, -0.5)  node[below] {{\text vertical ($\tV$)}} ;

\draw[thick] (0, 0)--(3, 0) node[below] {$B$};
\draw[thick] (1, 2)--(1, 1);
\draw[dotted] (1, 0)--(1, 3);
\draw[dotted] (2, 0)--(2, 3);

\end{tikzpicture}
\label{figcla}
\caption{The dotted lines are the fibers of $\pi x$ and $\pi y$. A pair $(x, y)$ is {\it horizontal} if it is ``pararell'' to the base, {\it vertical} if they are in the same fiber, and {\it tilted} otherwise. For a precise definition, see Definition \ref{defcla}. We abbreviate them to symbols $\tH, \tT, \tV$ in the following.}
\end{figure}

In the remained part of this article, we show the isomorphism of magnitude homology in Section \ref{hlgy}, and show the equvalence of magnitude homotopy type in Section \ref{htpytype}. The Section \ref{appendix} is an appendix section in which we show the equivalence of definions of the magnitude homotopy type.

\subsubsection*{Acknowledgements}
Y. A. was supported by JSPS KAKENHI 24K16927. Y. T. was supported by JST SPRING JPMJSP2119.
 M. Y. was partially supported by JSPS KAKENHI JP22K18668 and JP23H00081.
\section{Isomorphism at homology level}\label{hlgy}
\subsection{magnitude homology}
\begin{df}
    Let $(X, d)$ be a metric space.
    \begin{enumerate}
        \item For $\ell \in \R_{\geq 0}$ and $n \in \Z_{\geq 0}$, we define 
        \[
        P_n^\ell(X) := \{(x_0, \dots, x_n) \in X^{n+1} \mid x_i\neq x_{i+1}, \sum_{i=0}^{n-1}d(x_i, x_{i+1}) = \ell\},
        \]
        and $P_n(X) := \cup_{\ell} P_n^\ell(X)$.
        \item For $x, y, z \in X$, we write $x \prec y \prec z$ if $d(x, z) = d(x, y) + d(y, z)$.
        \item The {\it magnitude chain complex} $(\MC^{\ell}_\ast(X), \del^{\ell}_\ast)$ is defined by  $\MC^{\ell}_n(X) = \Z P^{\ell}_n(X)$ and
    \[
    \del_n (x_0, \dots, x_n) := \sum_{x_{i-1}\prec x_i \prec x_{i+1}}(-1)^i(x_0, \dots, \hat{x}_i, \dots, x_n).
    \]
Its homology $\MH^\ell_\ast(X)$ is called the {\it magnitude homology of } $X$.
        
    \end{enumerate}
\end{df}

\subsection{metric fibration}

\begin{df}
    A Lipschitz map $\pi : E \too B$ is a {\it metric fibration} if it satisfies the following : for all $x \in E$ and $b \in B$, there uniquely exists $x^b \in \pi^{-1}b$ satisfying 
    \begin{enumerate}
        \item $d(x, x^b) = d(\pi x, b)$,
        \item $d(x, y) = d(x, x^b) + d(x^b, y)$ for all $y \in \pi^{-1}b$.
    \end{enumerate}
\end{df}

\begin{lem}
    Let $\pi : E \too B$ be a metric fibration. For $b, b' \in B$, a map $ \pi^{-1}b \too \pi^{-1}b' ; x \mapsto x^{b'}$ is an isomorphism of metric spaces.
\end{lem}
\begin{proof}
    \cite{L2} Lemma 2.3.10, \cite{A2} Lemma 3.4.
\end{proof}

\begin{df}\label{defcla}
\begin{enumerate}
    \item Let $\tS$ be a monoid freely generated by words {\tt h, v, t}. We denote the subset of $\tS$ that consists of $n$ words by $\tS_n$.
    \item For a metric fibration $\pi : E \too B$, we define a map $T : P_1(E) \too \tS_1$ by 
    \[
    T(x, x') = \begin{cases} \tH & \text{ if }d(x, x') = d(\pi x, \pi x'), \\
    \tV  & \text{ if }d(\pi x, \pi x') = 0, \\
    \tT  & \text{ if }0 < d(\pi x, \pi x')  < d(x, x').\end{cases}
    \]
    We extend this map to a map $T : P_n(E) \too \tS$ by $T(x_0, \dots, x_n) = T(x_0, x_1)\dots T(x_{n-1}, x_n)$.
    \item For ${\tt xy} \in \tS_2$ and ${\tt z} \in \tS_1$, we write $\del {\tt xy} = {\tt z}$ if there is a metric fibration $\pi : E \too B$ and $(x, y, z) \in P_2(E)$ satisfying that $x\prec y \prec z, T(x, y, z) = {\tt xy}$ and $T(x, z) = {\tt z}$. We also define $\{\del {\tt xy}\} = \{{\tt z} \in \tS_1 \mid \del {\tt xy} = {\tt z}\}$.
    
    \end{enumerate}
\end{df}

\begin{rem}
    The words $\tH, \tV, \tT$ are abbreviations of {\it horizontal, vertical} and {\it tilted} respectively.
\end{rem}

\begin{eg}\label{egT}
In the following figures, the graph on the left is $(I_2 \times I_2) \times I_3$, where $I_n$ is the graph with vertices $\{1, \dots, n\}$ and edges $\{\{i, i+1\}\mid 1 \leq i \leq n-1\}$, and the graph on the right is a non-trivial metric fibration over the complete graph $K_3$ with the fiber $I_2$. We have the following :
\begin{enumerate}
\item $\begin{cases}1 \prec 2 \prec 6, T(1, 2, 6) = \tH\tV, T(1, 6) = \tT, \\ 
1 \prec 5 \prec 6, T(1, 5, 6) = \tV\tH, T(1, 6) = \tT,\end{cases}$
        \item $\begin{cases}1 \prec 2 \prec 7, T(1, 2, 7) = \tH\tT, T(1, 7) = \tT, \\
 1 \prec 6 \prec 7, T(1, 6, 7) = \tT\tH, T(1, 7) = \tT,\end{cases}$
        \item $\begin{cases}1 \prec 5 \prec 10, T(1, 5, 10) = \tV\tT, T(1, 10) = \tT, \\
1 \prec 6 \prec 10, T(1, 6, 10) = \tT\tV, T(1, 10) = \tT,\end{cases}$
        \item  $1 \prec 6 \prec 11, T(1, 6, 11) = \tT\tT, T(1, 11) = \tT$,
        \item $1 \prec 2 \prec 3, T(1, 2, 3) = \tH\tH, T(1, 3) = \tH$,
        \item $a \prec e \prec f, T(a, e, f) = \tH\tH, T(a, f) = \tT$.
\end{enumerate}
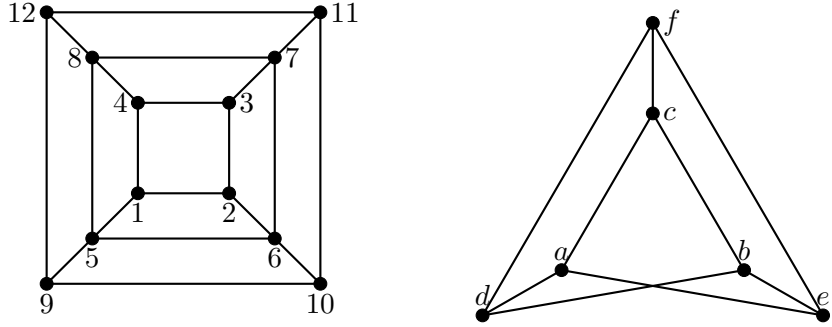
    \begin{figure}[htbp]
\centering
\begin{tikzpicture}[scale=1.2]
\filldraw[fill=black, draw=black] (0, 0) circle (2pt) node[below] {$5$} ;
\filldraw[fill=black, draw=black] (0, 2) circle (2pt) node[left] {$8$} ;
\filldraw[fill=black, draw=black] (2, 0) circle (2pt) node[below] {$6$} ;
\filldraw[fill=black, draw=black] (2, 2) circle (2pt) node[right] {$7$} ;
\filldraw[fill=black, draw=black] (0.5, 0.5) circle (2pt) node[below] {$1$} ;
\filldraw[fill=black, draw=black] (1.5, 0.5) circle (2pt) node[below] {$2$} ;
\filldraw[fill=black, draw=black] (1.5, 1.5) circle (2pt) node[right] {$3$} ;
\filldraw[fill=black, draw=black] (0.5, 1.5) circle (2pt) node[left] {$4$} ;
\filldraw[fill=black, draw=black] (-0.5, -0.5) circle (2pt) node[below] {$9$} ;
\filldraw[fill=black, draw=black] (2.5, -0.5) circle (2pt) node[below] {$10$} ;
\filldraw[fill=black, draw=black] (2.5, 2.5) circle (2pt) node[right] {$11$} ;
\filldraw[fill=black, draw=black] (-0.5, 2.5) circle (2pt) node[left] {$12$} ;

\draw[thick] (0, 0)--(2, 0);
\draw[thick] (2, 0)--(2, 2);
\draw[thick] (2, 2)--(0, 2);
\draw[thick] (0, 0)--(0, 2);
\draw[thick] (0.5, 0.5)--(1.5, 0.5);
\draw[thick] (1.5, 0.5)--(1.5, 1.5);
\draw[thick] (1.5, 1.5)--(0.5, 1.5);
\draw[thick] (0.5, 0.5)--(0.5, 1.5);

\draw[thick] (-0.5, -0.5)--(2.5, -0.5);
\draw[thick] (2.5, -0.5)--(2.5, 2.5);
\draw[thick] (2.5, 2.5)--(-0.5, 2.5);
\draw[thick] (-0.5, -0.5)--(-0.5, 2.5);

\draw[thick] (0, 0)--(0.5, 0.5);
\draw[thick] (1.5, 0.5)--(2, 0);
\draw[thick] (2, 2)--(1.5, 1.5);
\draw[thick] (0.5, 1.5)--(0, 2);

\draw[thick] (0, 0)--(-0.5, -0.5);
\draw[thick] (2.5, -0.5)--(2, 0);
\draw[thick] (2, 2)--(2.5, 2.5);
\draw[thick] (-0.5, 2.5)--(0, 2);

\end{tikzpicture}
\hspace{1cm}
\centering
\begin{tikzpicture}[scale=1.2]
\filldraw[fill=black, draw=black] (1, 1.732) circle (2pt) node[right] {$c$} ;
\filldraw[fill=black, draw=black] (2, 0) circle (2pt) node[above] {$b$} ;
\filldraw[fill=black, draw=black] (0, 0) circle (2pt) node[above] {$a$} ;
\filldraw[fill=black, draw=black] (1, 2.732) circle (2pt) node[right] {$f$} ;
\filldraw[fill=black, draw=black] (-0.866, -0.5) circle (2pt) node[above] {$d$} ;
\filldraw[fill=black, draw=black] (2.866, -0.5) circle (2pt) node[above] {$e$} ;

\draw[thick] (1, 1.732)--(2, 0);
\draw[thick] (1, 1.732)--(0, 0);
\draw[thick] (2, 0)--(-0.866, -0.5);
\draw[thick] (0, 0)--(-0.866, -0.5);
\draw[thick] (1, 1.732)--(1, 2.732);
\draw[thick] (2, 0)--(2.866, -0.5);
\draw[thick] (2.866, -0.5)--(0, 0);
\draw[thick] (1, 2.732)--(-0.866, -0.5);
\draw[thick] (1, 2.732)--(2.866, -0.5);

\end{tikzpicture}
\label{graphbdl}
\end{figure}
\end{eg}

\begin{lem}\label{formula}
For each ${\tt x}, {\tt y} \in \tS_1$, we have the following.
\begin{enumerate}
    \item $ \{\del {\tt xy}\} = \{\tV\} \Leftrightarrow \del {\tt xy} = \tV \Leftrightarrow {\tt xy} = \tV\tV$.
    \item $\{\del \tHV \}= \{\del \tV\tH \}=  \{\tT\}$, and $\{\del {\tt x}\tT\} = \{\del \tT{\tt x}\} = \{\tT\}$ for all ${\tt x} \in \tS_1$.
\item $\{\del \tH\tH\} = \{\tH, \tT\}$. 
\item For $(x, y, z) \in P_2(E) $ with $x\prec y \prec z$ and $T(x, y, z) = \tH\tH$, we have $T(x, z) = \tH$ if and only if $\pi x \prec \pi y \prec \pi z$.
\end{enumerate}
\end{lem}
\begin{proof}
    \begin{enumerate}
            \item Obviously we have $ \{\del {\tt xy}\} = \{\tV\} \Rightarrow \del {\tt xy} = \tV$. Also we have ${\tt xy} = \tV\tV \Rightarrow  \{\del {\tt xy}\} = \{\tV\}$. Hence it is enough to show that  $\del {\tt xy} = \tV$ implies ${\tt xy} = \tV\tV$. Let $(x, y, z) \in P_2(E)$ with $x \prec y \prec z$. We show that $T(x, z) = \tV$ implies $T(x, y) = T(y, z) = \tV$. If $T(x, z) =  \tV$, we have $\pi x = \pi z$, which implies 
        \begin{align*}
        d(x, y) + d(y, z) &= d(\pi x, \pi y) + d(x, y^{\pi x}) + d(\pi y, \pi z) + d(y^{\pi z}, z)\\
        &= d(x, y^{\pi x}) + d(y^{\pi x}, z) +  2d(\pi x, \pi y) \\
        &\geq d(x, z) + 2d(\pi x, \pi y).
        \end{align*}
       Since we have $x \prec y \prec z$, we obtain that $d(\pi x, \pi y) = d(\pi z, \pi y) = 0$, namely $T(x, y) = T(y, z) = \tV$.
        \item Note that we have $\del \tHV = \tT$ and $\del \tV\tH = \tT$ by Example \ref{egT} (1), and we also have $\neg(\del \tV\tH = \tH)$ and $\neg(\del \tV\tH = \tV)$ by the definition of the metric fibration. Hence we obtain $\{\del \tHV \}= \{\del \tV\tH\} = \{\tT\}$. Suppose that $T(x, y, z) = {\tt x}\tT$ for $(x, y, z) \in P_2(E), {\tt x} \in \tS_1$ and $x \prec y \prec z$. Then we have $d(x, z) = d(x, y) + d(y, z) > d(\pi x, \pi y) + d(\pi y, \pi z) \geq d(\pi x, \pi z) > 0$ by $T(y, z) = \tT$ and (1). Hence we obtain $T(x, z) = \tT$, and by Example \ref{egT} (2), (3) and (4), we obtain $\{\del {\tt x}\tT\} = \{\tT\}$. We can similarly show that  $\{\del \tT{\tt x}\} = \{\tT\}$.
        \item We have $\{\del \tH\tH\} \subset \{\tH, \tT\}$ by (1), and the inverse inclusion follows from Example \ref{egT} (5) and (6).
        \item By $T(x, y, z) = \tH\tH$ and $x\prec y \prec z$, we have 
        \[
        d(x, z) = d(x, y) + d(y, z) = d(\pi x, \pi y) + d(\pi y, \pi z).
        \]
        Hence $T(x, z) = \tH$ implies that $d(\pi x, \pi z) = d(x, z) = d(\pi x, \pi y) + d(\pi y, \pi z)$, and $\pi x \prec \pi y \prec \pi z$ implies that $d(x, z) = d(\pi x, \pi z)$.
    \end{enumerate}
\end{proof}

\subsection{a subcomplex $D_\ast^{\ell}(E)\subset \MC^{\ell}_\ast(E)$}
In the following, we construct a chain subcomplex $D^\ell_\ast(E) \subset \MC^{\ell}_\ast(E)$ that consists of  tuples of special types $P_n^{\ell, \tT}(E)$ and $P_n^{\ell, \tHV}(E)$.  We define the set $P_n^{\ell, \tT}(E) \subset P^\ell_n(E)$ as tuples containing {\it tilted pair} $(x_s, x_{s+1})$ earlier than {\it horizontal-vertical triple} $(x_t, x_{t+1}, x_{t+2})$ (namely $s+1 \leq t$). Dually, we define the set $P_n^{\ell, \tHV}(E) \subset P^\ell_n(E)$  as tuples containing  {\it horizontal-vertical triple} $(x_t, x_{t+1}, x_{t+2})$ earlier than {\it tilted pair} $(x_s, x_{s+1}) $ (namely $t+2 \leq s$). Formally we define them as follows.
\begin{df}\label{defd}
    For a metric fibration $\pi : E \too B$, we define subsets $P_n^{\ell, \tT}(E), P_n^{\ell, \tHV}(E) \subset P^{\ell}_n(E)$ by 
    \begin{align*}
        P_n^{\ell, \tT}(E) &:= \{x \in P^\ell_n(E) \mid Tx \in \tV^m\tH^{m'}\tT\tS \text{ for } m, m'\geq 0\}, \\
        P_n^{\ell, \tHV}(E) &:= \{x \in P^\ell_n(E) \mid Tx \in \tV^m\tH^{m'+1}\tV\tS \text{ for } m, m'\geq 0\}.
    \end{align*}
    We also define a submodule  $D^\ell_n(E) := \Z P_n^{\ell, \tT, \tHV}(E) \subset \MC^{\ell}_n(E)$, where $ P_n^{\ell, \tT, \tHV}(E) = P_n^{\ell, \tT}(E)\cup P_n^{\ell, \tHV}(E)$.
\end{df}


\begin{lem}\label{subchain}
    We have $\del_n x \in D^\ell_{n-1}(E)$ for $x \in P_n^{\ell, \tT, \tHV}(E)$. Namely, $D^\ell_\ast(E) \subset \MC^{\ell}_\ast(E)$ is a chain subcomplex.
\end{lem}
\begin{proof}
     It follows from Lemma \ref{formula}.
\end{proof}

\begin{prop}\label{chainisom}
    Let $\pi : E \too B$ be a metric fibration.  We fix $b \in B$ and $F := \pi^{-1}b$. Then we have an isomorphism of chain complexes 
    \[
    \MC^{\ell}_{\ast}(E)/D_\ast^{\ell}(E)\cong \bigoplus_{\ell_{\tV} + \ell_{\tH} = \ell} \MC^{\ell_{\tV}}_\ast(F)\otimes \MC^{\ell_{\tH}}_\ast(B).
    \]
    
\end{prop}

\begin{proof}
    Note that the module $\MC^{\ell}_{n}(E)/D_n^{\ell}(E)$ is freely generated by tuples $x \in P^{\ell}_{n}(E)$ with $Tx = \tV^m\tH^{n-m}$ for some $0 \leq m \leq n$. For each $n\geq 0$,  we define a homomorphism $\varphi_n :  \MC^{\ell}_{n}(E)/D_n^{\ell}(E) \too \bigoplus_{\substack{\ell_{\tV} + \ell_{\tH} = \ell\\ m\geq 0}} \MC^{\ell_{\tV}}_{m}(F)\otimes \MC^{\ell_{\tH}}_{n-m}(B)$ by 
    \[
    \varphi_n(x_0, \dots, x_n) = (x_0^{b}, \dots, x_m^b)\otimes (\pi x_m, \dots, \pi x_n),
    \]
    where we suppose that $T(x_0, \dots, x_n) = \tV^m\tH^{n-m}$. This homomorphism has an inverse $\psi_n$ defined by 
    \begin{align*}
    \psi_n\left((f_0, \dots, f_{m})\otimes (b_0, \dots, b_{n-m})\right) = (f_0^{b_0}, \dots, f_{m}^{b_0}, f_{m}^{b_0b_1}, f_{m}^{b_0b_1b_2}, \dots, f_{m}^{b_0\dots b_{n-m}}),
    \end{align*}
    where we denote a point $(f_{m}^{b_0})^{b_1}$ by $f_{m}^{b_0b_1}$ and similarly for further iterations.
    Hence it reduces to show that $\varphi_\ast$ is a chain map. We denote the boundary operator on $\MC^{\ell}_{\ast}(E)/D_\ast^{\ell}(E)$ induced from $\del^{\ell}_\ast$ by $[\del^{\ell}]_\ast$ in the following. For $(x_0, \dots, x_n) \in \MC^{\ell}_{n}(E)/D_n^{\ell}(E)$ with $T(x_0, \dots, x_n) = \tV^m\tH^{n-m}$, we have 
    \begin{align*}
    [\del^{\ell}]_n(x_0, \dots, x_n) &= \sum_{\substack{x_{i-1}\prec x_i \prec x_{i+1}\\ T(x_{i-1}, x_{i+1}) \neq \tT}}(-1)^i(x_0, \dots, \hat{x}_i, \dots, x_n) \\
    &= \sum_{\substack{x_{i-1}\prec x_i \prec x_{i+1}\\ 1 \leq i \leq m-1}}(-1)^i(x_0, \dots, \hat{x}_i, \dots, x_m, \dots, x_n) \\
    &+ \sum_{\substack{\pi x_{i-1}\prec \pi x_i \prec \pi x_{i+1}\\ m+1 \leq i \leq n-1}}(-1)^i(x_0, \dots, x_m, \dots,  \hat{x}_i, \dots,  x_n),
    \end{align*}
    by Lemma \ref{formula} (2) and (4). Hence we obtain that 
    \begin{align*}
    \varphi_{n-1}[\del^{\ell}]_n(x_0, \dots, x_n) &= \sum_{\substack{x^b_{i-1}\prec x^b_i \prec x^b_{i+1}\\ 1 \leq i \leq m-1}}(-1)^i(x^b_0, \dots, \hat{x}_i^b, \dots, x^b_m) \otimes (\pi x_m, \dots, \pi x_n) \\
    &+ \sum_{\substack{\pi x_{i-1}\prec \pi x_i \prec \pi x_{i+1}\\ m+1\leq i \leq n-1}}(-1)^i(x^b_0, \dots, x^b_m) \otimes (\pi x_m, \dots,  \widehat{\pi x}_i, \dots, \pi x_n).
    \end{align*}
     On the other hand, for $\varphi_n(x_0, \dots, x_n) = (x^b_0, \dots, x^b_{m})\otimes (\pi x_m, \dots, \pi x_n) \in \MC^{\ell_{\tV}}_{m}(F)\otimes \MC^{\ell_{\tH}}_{n-m}(B)$, we have
    \begin{align*}
    (\del_m^{\ell_{\tV}}\otimes \del_{n-m}^{\ell_{\tH}})\varphi_n(x_0, \dots, x_n)  &= \sum_{\substack{x^b_{i-1}\prec x^b_i \prec x^b_{i+1}\\ 1 \leq i \leq m-1}}(-1)^i(x^b_0, \dots, \hat{x}_i^b, \dots, x^b_{m})\otimes (\pi x_m, \dots, \pi x_n) \\
    &+  \sum_{\substack{\pi x_{i-1}\prec \pi x_i \prec \pi x_{i+1}\\ m+1\leq i \leq n-1}}(-1)^{i}(x^b_0, \dots, x^b_{m})\otimes (\pi x_m, \dots, \widehat{\pi x}_i, \dots,   \pi x_n).
    \end{align*}
    Thus we obtain that $\varphi_{n-1}[\del^{\ell}]_n =  (\del_m^{\ell_{\tV}}\otimes \del_{n-m}^{\ell_{\tH}})\varphi_n$.
\end{proof}

\subsection{Algebraic Morse Theory}\label{alg1}
We recall the algebraic Morse theory studied in~\cite{Sk}. Let $C_{\ast} = (C_{\ast}, \partial_{\ast})$ be a chain complex with a decomposition $C_{k} = \bigoplus_{a \in I_{n}} C_{n, a}$ and $C_{n, a} \cong \Z$ for each $k$. For $a \in I_{n+1}$ and $b \in I_n$, let $f_{ab}\colon C_{n+1, a} \too  C_{n, b}$ be the composition
$
C_{n+1, a} \hookrightarrow C_{n+1} \xrightarrow{\partial_{n+1}}  C_{n} \twoheadrightarrow C_{n, b}.
$ We define a directed graph $\Gamma_{C_{\ast}}$ with vertices $\coprod_{n} I_{n}$ and directed edges $\{a \to b \mid f_{ab} \neq 0\}$. We recall terminologies on the matching.  
\begin{df}
\begin{enumerate}
\item A {\it matching} $M$ of a directed graph $\Gamma$ is a subset of directed edges $M \subset E(\Gamma)$ such that each two distinct edges in $M$ have no common vertices. 
\item For a matching $M$ of a directed graph, vertices that are not the endpoints of any edges in $M$ are called {\it critical}. 
\item For a matching $M$ of a directed graph $\Gamma$, we define a new directed graph $\Gamma^{M}$ by inverting the direction of all edges in $M$.
\end{enumerate}
\end{df}
\begin{df}\label{morsematching}
A matching $M$ on $\Gamma_{C_{\ast}}$  is called a {\it Morse matching} if it satisfies the following.
\begin{enumerate}
\item  $f_{ab}$ is an isomorphism if $a \to b \in M$. 
\item  $\Gamma_{C_{\ast}}^{M}$ is acyclic, that is, there are no closed paths in $\Gamma_{C_{\ast}}^{M}$ of the form
$
 a_{1}  \to b_{1} \to \dots \to b_{p-1} \to a_{p}=a_1
$
with $a_{i} \in I_{n+1}$ and $b_{i}\in I_{n}$ for some $p$.

\end{enumerate}
 \end{df}

 For a matching $M$ on $\Gamma_{C_{\ast}}$, we denote the subset of $I_n$ that consists of critical vertices by $\mathring{I}_{n}$. 

 \begin{prop}[\cite{Sk}]\label{morsefact}
 For a Morse matching $M$ on $\Gamma_{C_{\ast}}$, we have a chain complex $(\mathring{C}_{n} = \bigoplus_{a \in \mathring{I}_{n}} C_{n, a}, \mathring{\del}_\ast)$ that is  homotopy equivalent to $(C_\ast, \del_\ast)$.
 \end{prop}

 \subsection{matching on $D_\ast^\ell(E)$}
We apply algebraic Morse theory to the chain complex $(D^\ell_\ast(E), \del^{\ell}_\ast)$ with the decomposition $D^{\ell}_n(E) = \bigoplus_{a \in P^{\ell, \tT, \tHV}_n(E)} D_{n, a} $ and $D_{n, a} \cong \Z$. For $a = (x_0, \dots, x_{n+1}) \in P^{\ell, \tT, \tHV}_{n+1}(E)$ and $b \in P^{\ell, \tT, \tHV}_{n}(E)$, we write $b = \del^{\ell}_{n+1, i}a$ if $b = (x_0, \dots, \hat{x}_i, \dots, x_{n+1})$. It is immediately verified that $f_{ab}$ is an isomorphism for $a \in P^{\ell, \tT, \tHV}_{n+1}(E)$ and $b \in P^{\ell, \tT, \tHV}_{n}(E)$ if and only if $b = \del^{\ell}_{n+1, i}a$ for some $i$.

 \begin{df}
 \begin{enumerate}
\item For $a = (x_0, \dots, x_n) \in P^{\ell, \tT}_n(E)$ with $Ta \in \tV^m\tH^{m'}\tT\tS$, we define 
\[
a^{\tH\tV} := (x_0, \dots, x_{m+m'}, x_{m+m'}^{\pi x_{m+m'+1}},  x_{m+m'+1}, \dots, x_n).
\]
\item For $(x_0, \dots, x_n) \in P^{\ell}_n(E)$, we define 
\[
|(x_0, \dots, x_n)| := \sum_{T(x_i, x_{i+1})= \tV} i.
\]
\end{enumerate}
\end{df}
Namely, we obtain a tuple $a^{\tH\tV}$ by filling the gap of the first tilted part of $a$. The filled part becomes horizontal-vertical triple.
\begin{lem}\label{acyclic}
    Let $a_1 \neq a_2 \in P^{\ell, \tT}_n(E)$. If $a_2 = \del^{\ell}_{n+1, i}a_1^{\tHV}$ for some $i$,  then we have $|a_1^{\tHV}| < |a_2^{\tHV}|$.
\end{lem}
\begin{proof}
    Suppose that $Ta_1 = \tV^m\tH^{m'}\tT{\tt xw}$ for some ${\tt x} \in \tS_1$ and ${\tt w} \in \tS$. Then we have $Ta_1^\tHV =  \tV^m\tH^{m'+1}\tV{\tt xw}$. If we have $\del^{\ell}_{n+1, i}a_1^{\tHV} = a_2 \in P^{\ell, \tT}_n(E)$, then we should have  
    \[
    Ta_2 \in \{ \tV^{m-1}\tT\tH^{m'}\tV{\tt xw},  \tV^m\tH^{m''}\tT\tH^{m'-m''-2}\tV{\tt xw},  \tV^m\tH^{m'+1}\tT{\tt w}\},
    \]
    by Lemma \ref{formula}. In each case, we have 
    \[
    Ta_2^\tHV \in \{\tV^{m-1}\tHV\tH^{m'}\tV{\tt xw},  \tV^m\tH^{m''+1}\tV\tH^{m'-m''-2}\tV{\tt xw},  \tV^m\tH^{m'+2}\tV{\tt w} \}
    \]
    respectively. In all cases, we have $|a_1^{\tHV}| < |a_2^{\tHV}|$. 
\end{proof}
We define a matching $M$ on $D_\ast^\ell(E)$ by 
\[
M = \{f_{a^{\tHV}a} : a^{\tHV} \to a \mid a \in P^{\ell, \tT}_n(E) \}.
\]
This is apparently a matching, and is also acyclic by Lemma \ref{acyclic}. Further, there is no critical vertex in $\Gamma_{D_\ast^\ell(E)}$. Thus we obtain the following by Proposition \ref{morsefact}.

\begin{prop}\label{dcontractible}
    The chain complex $D_\ast^\ell(E)$ is contractible.
\end{prop}

\begin{cor}\label{main1}
    Let $\pi : E \too B$ a metric fibration, and let $F$ be its fiber. For $\ell>0$, we have a homotopy equivalence and an isomorphism
    \[
    \MC^{\ell}_\ast(E) \simeq \MC^\ell_\ast(E)/D^\ell_\ast(E) \cong \bigoplus_{\ell_\tV + \ell_\tH = \ell} \MC^{\ell_\tV}_\ast(F)\otimes \MC^{\ell_\tH}_\ast(B).
    \]
\end{cor}
\begin{proof}
    It follows from Propositions \ref{chainisom}, \ref{dcontractible} and the fact that each quasi-isomorphism between levelwise free chain complexes is induced from a homotopy equivalence.
\end{proof}
\begin{rem}
    Note that, by Corollary \ref{main1}, we reprove  the K\"{u}nneth theorem in \cite{HW} Proposition 8.4, namely $ \MH^{\ell}_\ast(F\times B) \cong H_\ast(\bigoplus_{\ell_\tV + \ell_\tH = \ell} \MC^{\ell_\tV}_\ast(F)\otimes \MC^{\ell_\tH}_\ast(B)).$
\end{rem}

\section{Equivalence of magnitude homotopy type}\label{htpytype}

\subsection{$\Delta$-set}

We denote the category of finite ordinals $\{0 < 1< \dots < n\} = :[n]$'s and order preserving maps between them by $\Delta$. We define maps $\delta_{n,i} : [n-1] \too [n]$ and $\sigma_{n,i} : [n+1] \too [n]$ for $0 \leq i \leq n$ by $\delta_{n, i} j = \begin{cases}
        j & j < i, \\
        j+1 & j \geq i,
    \end{cases}$ and $ \sigma_{n, i} j = \begin{cases}
        j & j \leq i, \\
        j-1 & j > i.\end{cases}$ We abbreviate them to $\delta_i$ and $ \sigma_i$. Note that all order preserving map $f : [m] \too [n]$ can be uniquely decomposed as a composition of order preserving maps $f = \phi_1(f)\phi_2(f)$ such that $\phi_1(f)$ is injective and $\phi_2(f)$ is surjective. Also, we can decompose $\phi_1(f)$ and $\phi_2(f)$ into compositions of $\delta_i$'s and $\sigma_i$'s respectively.
\begin{df}
    A family of sets $X_\bullet = \{X_n\}_{n\geq 0}$ equipped with maps $d_i : X_n \too X_{n-1} (0 \leq i \leq n)$ is called a {\it $\Delta$-set} if it satisfies $d_id_j = d_{j-1}d_i$ for $i<j$. Equivalently, a $\Delta$-set is a functor $\Delta_{\sf inj}^{\rm op} \too \Set$, where $\Delta_{\sf inj}$ is the category of finite ordinals and order preserving injections that are generated from $\delta_i$'s. We define the category of $\Delta$-sets by $\Delta\Set := \Set^{\Delta_{\sf inj}^{\rm op}}$.
\end{df}

 Note that the inclusion $j : \Delta_{\sf inj} \too \Delta$ induces a functor $j^\ast : \sSet \too \Delta\Set$. Namely, for a simplicial set $S_\bullet$, we can obtain a $\Delta$-set $j^\ast S_\bullet$ by forgetting the degeneracy maps. The functor $j^\ast$ has the left adjoint (\cite{RS} Theorem 1.7)  $j_! : \Delta\Set \too \sSet$ defined by 
 \[
 (j_!X_\bullet)_n = \{(p, f) \mid p \in X_{n-k}, f : [n] \twoheadrightarrow [n-k] \in \Delta , 0\leq k \leq n\}.
 \]
The structure maps $d_i : (j_!X_\bullet)_n \too (j_!X_\bullet)_{n-1}, s_i : (j_!X_\bullet)_n \too (j_!X_\bullet)_{n+1} $ for $ 0\leq i \leq n$ are defined by 

 \begin{align*}
     d_i(p, f) &= (\big(\phi_1(f\delta_i)\big)^\ast p, \phi_2(f\delta_i)),\\
     s_i(p, f) &= (p, f\sigma_i),
 \end{align*}
where we use the following composition and factorization of maps: 
\begin{equation*}
    \xymatrix{
    [n-1] \ar[r]^-{\delta_i} \ar[dr]_-{\phi_2(f\delta_i)} & [n] \ar[r]^-{f} & [n-k] \\
     & [m]  \ar[ur]_-{\phi_1(f\delta_i)}&.
    }
\end{equation*}
\begin{eg}
    \begin{enumerate}
        \item For a metric space $X$, $\ell \in \R_{\geq 0}$ and $n \in \Z_{\geq 0}$, we define ${\sf m}^{\ell}_n(X) := P^{\ell}_n(X)\cup\{\ast\}$. We also define maps $d_i : {\sf m}^{\ell}_n(X) \too {\sf m}^{\ell}_{n-1}(X)$ for $0 \leq i \leq n$ by
        \begin{align*}
            d_i(\ast) &= \ast, \\
            d_i(x_0, \dots, x_n) &= \begin{cases}
                (x_0, \dots, \hat{x}_i, \dots, x_n) & \text{if } x_{i-1}\prec x_i \prec x_{i+1}, 1\leq i \leq n-1, \\
                \ast & \text{otherwise.}
            \end{cases}
        \end{align*}
Then it is immediate to verify that ${\sf m}^{\ell}_\bullet(X)$ is a $\Delta$-set.

\item For a metric space $X$, we denote Hepworth and Willerton's simplicial set (\cite{HW} Definition 8.1) by ${\sf M}^\ell_\bullet(X)$. That is defined by 
\[
{\sf M}^\ell_n(X) = \{(x_0, \dots, x_n) \in X^{n+1}\mid \sum_{i=0}^{n-1}d(x_{i}, x_{i+1}) = \ell\}\cup\{\ast\},
\]
for $\ell\in\R_{\geq 0}$ and $n \in \Z_{\geq 0}$. The maps $d_i$'s are defined by the same formula as those of ${\sf m}^\ell_\bullet$, and $s_i$'s are defined by $s_i(x_0, \dots, x_n) = (x_0, \dots, x_i, x_i, \dots, x_n)$ and $s_i(\ast) = \ast$.

\item  For a point $\ast \in \sSet$, defined by $\ast_n = \{\ast\}$, we have
\[
(j_!j^\ast \ast)_n \cong \{f : [n] \twoheadrightarrow [n-k] \mid 0 \leq k \leq n\}, 
\]
and $d_if = \phi_2(f\delta_i), s_i f = f\sigma_i$ for $f : [n] \twoheadrightarrow [n-k]$. Note that the non-degenerate simplices of $(j_!j^\ast \ast)_\bullet$ are only identities ${\rm id}_{[n]}$, and its geometric realaization $|(j_!j^\ast \ast)_\bullet|$ is $S^\infty$.

\item For a metric space $X$ and $\ell \in \R_{\geq 0}$, we define a simplicial set $\wt{\sf M}^\ell_\bullet(X)$ by \[
\wt{\sf M}^\ell_n(X) = \{(x_0, \dots, x_n) \in X^{n+1}\mid \sum_{i=0}^{n-1}d(x_{i}, x_{i+1}) = \ell\}\cup\{f : [n] \twoheadrightarrow [n-k] \mid 0 \leq k \leq n\}.
\]
We define 
\begin{align*}
            d_i(f) &= \phi_2(f\delta_i), \\
            d_i(x_0, \dots, x_n) &= \begin{cases}
                (x_0, \dots, \hat{x}_i, \dots, x_n) & \text{if } x_{i-1}\prec x_i \prec x_{i+1}, 1\leq i \leq n-1, \\
                {\rm id}_{[n-1]}& \text{otherwise.}
            \end{cases}
        \end{align*}

and 
\begin{align*}
            s_i(f) &= f\sigma_i, \\
            s_i(x_0, \dots, x_n) &= (x_0, \dots, x_i, x_i, \dots, x_n).
        \end{align*}
    \end{enumerate}
\end{eg}

\begin{prop}\label{jmM}
    We have $j_!{\sf m}^\ell_\bullet(X) \cong \wt{\sf M}^\ell_\bullet (X)$.
\end{prop}

\begin{proof}
In the following, we denote the maps $j_!{\sf m}^\ell_n(X) \too j_!{\sf m}^\ell_m(X)$ and  $\wt{\sf M}^{\ell}_n(X) \too \wt{\sf M}^{\ell}_m(X)$ induced from a map $f : [m] \too [n]$ by $f^{\sf m}$ and $f^{\sf M}$ respectively. We also denote the structure maps $d_i, s_i$'s of $j_!{\sf m}^\ell_\bullet(X)$ and $\wt{\sf M}^{\ell}_\bullet(X)$ by $d_i^{\sf m}, s_i^{\sf m}$ and $d_i^{\sf M}, s_i^{\sf M}$'s respectively. We define a map $F_n : (j_!{\sf m}^\ell_\bullet(X))_n \too \wt{\sf M}^\ell_n (X)$ by
\[
F_n (p, f) = \begin{cases} f & p= \ast \\
f^{\sf M}p & p \neq  \ast,\end{cases}
\]
where we identify an element $p \in P^{\ell}_{n-k}(X) \subset {\sf m}^\ell_{n-k}(X)$ with an element $p \in \wt{\sf M}^\ell_{n-k} (X)$. This map is obviously a bijection, hence it reduces to show that this defines a morphism of simplicial sets. Now we have 

\begin{align*}
F_{n+1} s_i^{\sf m}(p, f) &= F_{n+1}(p, f\sigma_i)=\begin{cases} f\sigma_{i} & p= \ast \\
s_i^{\sf M}f^{\sf M}p & p\neq \ast\end{cases} =s_i^{\sf M}F_{n}(p,f).
\end{align*}
We also have
\begin{align*}
F_{n-1} d_i^{\sf m}(p, f) &= F_{n-1}(\phi_1^{\sf m} p, \phi_2) =\begin{cases} \phi_2 & \phi_1^{\sf m} p= \ast \\
\phi_2^{\sf M}\phi_1^{\sf M} p & \phi_1^{\sf m} p\neq \ast,\end{cases}
\end{align*}
where we abbreviate $\phi_1(f\delta_i), \phi_2(f\delta_i)$ to  $\phi_1, \phi_2$ respectively, and we identify $\phi^{\sf m}_1p \in {\sf m}^\ell_{\bullet}(X)$ with $\phi^{\sf M}_1p \in \wt{\sf M}^\ell_{\bullet} (X)$. Also, we have
 \begin{align*}
 d_i^{\sf M}F_n(p, f) &= \begin{cases} d_i^{\sf M}f & p= \ast, \\
 d_i^{\sf M}f^{\sf M} p & p \neq \ast,\end{cases}\\
 &= \begin{cases} \phi_2 & p= \ast, \\
 \phi_2^{\sf M}\phi_1^{\sf M} p & p\neq \ast,\end{cases}\\
 &= \begin{cases} \phi_2 & p= \ast, \\
 \phi_2 & p \neq \ast, \phi_1^{\sf m} p =\ast ,\\
 \phi_2^{\sf M}\phi_1^{\sf M} p & p\neq \ast, \phi_1^{\sf m} p \neq \ast, \\
\end{cases}\\
&=\begin{cases} \phi_2 & \phi_1^{\sf m} p= \ast, \\
\phi_2^{\sf M}\phi_1^{\sf M} p & \phi_1^{\sf m} p\neq \ast.\end{cases}
\end{align*}
Hence $F_\bullet$ is an isomorphism of simplicial sets.

\end{proof}

\begin{prop}\label{wtmm}
    We have a homotopy equivalence $|\wt{\sf M}^{\ell}_\bullet(X)| \simeq |{\sf M}^{\ell}_\bullet(X)|$.
\end{prop}
\begin{proof}
Obviously we have an inclusion $j_!j^\ast\ast \too \wt{\sf M}^{\ell}_\bullet(X)$, and its quotient map $\wt{\sf M}^{\ell}_\bullet(X)\too {\sf M}^{\ell}_\bullet(X)$. Hence it induces a sequence $|j_!j^\ast\ast| \too |\wt{\sf M}^{\ell}_\bullet(X)|\too |{\sf M}^{\ell}_\bullet(X)|$. Since $|j_!j^\ast\ast| \simeq S^\infty$ is a subcomplex of $|\wt{\sf M}^{\ell}_\bullet(X)|$, we conclude that $|\wt{\sf M}^{\ell}_\bullet(X)| \simeq |{\sf M}^{\ell}_\bullet(X)|$.
\end{proof}









\subsection{$D^\ell_\bullet(E) \subset {\sf m}^\ell_\bullet(E)$}

\begin{df}
    For a metric fibration $\pi : E \too B$, we define a $\Delta$-subset $D^\ell_\bullet(E) \subset {\sf m}^\ell_\bullet(E)$ by $D^\ell_n(E) = P^{\ell, \tT, \tHV}_n(E)\cup\{\ast\}$ for $\ell \in \R_{\geq 0}$.
\end{df}
We can verify that $D^\ell_\bullet(E)$ is indeed a $\Delta$-set by Lemma \ref{formula}.

\begin{lem}\label{dcont2}
    $|j_!D^\ell_\bullet(E)|$ is contractible.
\end{lem}
\begin{proof}
    By the same argument as the proof of Proposition \ref{wtmm}, $|j_!D^\ell_\bullet(E)|$ is homotopy equivalent to the geometric realization of a simplicial subset $K_\bullet \subset {\sf M}^\ell_\bullet(E)$ generated from the family of sets $P^{\ell, \tT, \tHV}_\bullet(E)$. Since the non-degenerate simplices of $K_\bullet$ are elements of $P^{\ell, \tT, \tHV}_n(E)$'s, the chain complex $C_\ast K$ is homotopy equivalent to the chain complex $D^\ell_\ast(E)$ of Definition \ref{defd}, which is contractible. Therefore it reduces to show that $|K_\bullet|$ is simply connected. Recall that the fundamental groupoid $\Pi_1 |K_\bullet|$ is equivalent to the fundamental groupoid $\Pi_1 K_\bullet$, whose objects are vertices of $K_\bullet$ and morphisms are generated by edges of $K_\bullet$ with the identification $d_0\sigma d_2\sigma \sim d_1\sigma$ for $\sigma \in K_2$. Now $\Pi_1 K_\bullet$ has only one object, and each morphism is a sequence of tuples $(x_0, x_1)$ with $T(x_0, x_1) = \tT$. Since we have $(x_0, x_1)=d_1(x_0, x_1)^{\tHV}\sim d_0(x_0, x_1)^{\tHV}d_2(x_0, x_1)^{\tHV} \sim \ast$, this groupoid is a trivial group.
\end{proof}
\begin{prop}\label{heq2}
    We have a homotopy equivalence $|j_!{\sf m}^\ell_\bullet(E)| \simeq |j_!{\sf m}^\ell_\bullet(E)/j_!D^\ell_\bullet(E)|$.
\end{prop}
\begin{proof}
   Same as Proposition \ref{wtmm}.
\end{proof}

\begin{prop}\label{deltaisom}
    We have ${\sf m}^\ell_\bullet(E)/D^\ell_\bullet(E) \cong {\sf m}^\ell_\bullet(F\times B)/D^\ell_\bullet(F\times B)$, where $F = \pi^{-1}b$ for a fixed $b\in B$.
\end{prop}
\begin{proof}
    We define a map $\varphi_\bullet : {\sf m}^\ell_\bullet(E)/D^\ell_\bullet(E) \too {\sf m}^\ell_\bullet(F\times B)/D^\ell_\bullet(F\times B)$ by $\varphi_n(\ast) = \ast$ and 
     \begin{align*}
    &\varphi_n(x_0, \dots, x_n)  \\
    &=((x_0^{b}, \pi x_0), \dots, (x_i^b, \pi x_i),  \dots, (x_m^b, \pi x_m),(x_m^b, \pi x_{m+1}), \dots, (x_m^b, \pi x_{m+j}) \dots,  (x_m^b, \pi x_{n})),
    \end{align*}
    where we suppose that $T(x_0, \dots, x_n) = \tV^m\tH^{n-m}$. This map has an inverse $\psi_\bullet$ defined by 
    \begin{align*}
    \psi_n((f_0, b_0), \dots, (f_{m}, b_0), (f_m, b_1), \dots, (f_m, b_{n-m}))= (f_0^{b_0}, \dots, f_{m}^{b_0}, f_{m}^{b_0b_1}, f_{m}^{b_0b_1b_2}, \dots, f_{m}^{b_0\dots b_{n-m}}).
    \end{align*}
    Hence it reduces to show that $\varphi_\bullet$ is a morphism of $\Delta$-sets, but it can be verified in the same manner as Propposition \ref{chainisom}. 
\end{proof}

\begin{cor}\label{main2}
    Let $\pi : E \too B$ be a metric fibration and let $F$ be its fiber. Then we have a homotopy equivalence $|{\sf M}^{\ell}_\bullet(E)| \simeq |{\sf M}^{\ell}_\bullet(F\times B)|$. 
\end{cor}
\begin{proof}
    We have homotpy equivalences
    \begin{align*}
    &|{\sf M}^{\ell}_\bullet(E)| \simeq |j_! {\sf m}^{\ell}_\bullet(E)|  \simeq  | j_!{\sf m}^{\ell}_\bullet(E)/j_!D^\ell_\bullet(E)| \\
    &\cong  |j_!{\sf m}^{\ell}_\bullet(F\times B)/j_!D^\ell_\bullet(F\times B)|   \simeq  | j_!{\sf m}^{\ell}_\bullet(F\times B)|  \simeq  | {\sf M}^{\ell}_\bullet(F\times B)|,
    \end{align*}
by Propositions \ref{jmM}, \ref{wtmm}, \ref{heq2} and \ref{deltaisom}. Note that $j_!$ commutes with quotients since it is a left adjoint.
\end{proof}
From Tajima and Yoshinaga's K\"{u}nneth theorem for magnitude homotopy type (\cite{TY} Theorem 4.27) together with the coincidence of two definitions of magnitude homotopy types (Proposition \ref{TYprof}), we have the following.
\begin{cor}\label{main3}
    Let $\pi : E \too B$ be a metric fibration and let $F$ be its fiber. Then we have a homotopy equivalence $|{\sf M}^{\ell}_\bullet(E)| \simeq \bigvee_{\ell_\tV + \ell_\tH = \ell}|{\sf M}^{\ell_\tV}_\bullet(F)| \wedge |{\sf M}^{\ell_\tH}_\bullet(B)|$. 
\end{cor}

\section{Appendix}\label{appendix}

In this appendix, we prove the following proposition which is stated in \cite{TY} without a proof.

\begin{prop}\label{TYprof}
   Let $X$ be a metric space and $\ell\in \R_{\geq 0}$. Tajima and Yoshinaga's magnitude homotopy type $\mathcal{M}^\ell(X)$ is homeomorphic to the geometric realization $|{\sf M}^\ell_\bullet(X)|$ of Hepworth and Willerton's simplicial set ${\sf M}^\ell_\bullet(X)$. 
\end{prop}

Recall from \cite{TY} that the CW complex $\mathcal{M}^\ell(X)$ is defined as the quotient $|\Delta{\rm Cau}^\ell(X)|/|\Delta'{\rm Cau}^\ell(X)|$ of the geometric realization of simplicial complexes $\Delta{\rm Cau}^\ell(X)$ and $\Delta'{\rm Cau}^\ell(X)$. Here, the simplicial complex $\Delta{\rm Cau}^\ell(X)$ is the order complex of the poset ${\rm Cau}^\ell(X) = \coprod_{a, b \in X} {\rm Cau}^\ell(X ; a, b)$ defined by 
\[
{\rm Cau}^\ell(X ; a, b) = \{(x, t) \in X \times [0, \ell] \mid d(a, x) \leq t, d(x, b) \leq \ell - t\},
\]
where $(x, t) \leq (x', t')$ if and only if $d(x, x') \leq t'-t$. Then the simplicial complex $\Delta{\rm Cau}^\ell(X) = \coprod_{a, b \in X}\Delta{\rm Cau}^\ell(X ; a, b)$ is defined by
\[
\Delta{\rm Cau}^\ell(X ; a, b) = \{\{(x_0, t_0), \dots, (x_n, t_n)\} \mid d(x_i, x_{i+1}) \leq t_{i+1}-t_i \text{ for} -1 \leq i \leq n \},
\]
where we put $x_{-1} = a, x_{n+1} = b, t_{-1} = 0, t_{n+1} = \ell$. Since we can extend each partial order to a total order, the simplicial complex $\Delta{\rm Cau}^\ell(X ; a, b)$ can be considered as an ordered simplicial complex, and each face of it can be expressed as a tuple $((x_0, t_0), \dots, (x_n, t_n))$ which is not just a set of points $\{(x_0, t_0), \dots, (x_n, t_n)\}$. The simplicial subcomplex $\Delta'{\rm Cau}^\ell(X) = \coprod_{a, b \in X} \Delta{\rm Cau}^\ell(X ; a, b)$ is defined by
\[
\Delta'{\rm Cau}^\ell(X ; a, b) = \{((x_0, t_0), \dots, (x_n, t_n)) \in \Delta{\rm Cau}^\ell(X ; a, b) \mid \sum_{i=0}^{n-1}d(x_i, x_{i+1}) <\ell \},
\]
which is also ordered. Here we note that we have $d(x_i, x_{i+1}) = t_{i+1}-t_i$ for all $-1\leq i \leq n$ if and only if $\sum_{i=0}^{n-1}d(x_i, x_{i+1}) = \ell$ by Proposition 4.2 of \cite{TY}.
\begin{proof}[Proof of Proposition \ref{TYprof}]
    Note first that each ordered simplicial complex $X$ can be turned into a $\Delta$-set $\overline{X}$ in a natural manner, and we obtain a simplicial set $j_!\overline{X}$. Obviously, the geometric realization of the ordered simplicial complex $X$ is homeomorphic to the geometric realization $|j_!\overline{X}|$ by the definitions. Also, for a pair $Y \subset X$ of ordered simplicial complexes, we have $|X|/|Y| \cong |j_!\overline{X}|/|j_!\overline{Y}| \cong |j_!\overline{X}/j_!\overline{Y}|$. Hence we have 
    \begin{align*}
    \mathcal{M}^\ell(X) &= |\Delta{\rm Cau}^\ell(X)|/|\Delta'{\rm Cau}^\ell(X)| \\
    &\cong \bigvee_{a, b}|\Delta{\rm Cau}^\ell(X; a, b)|/|\Delta'{\rm Cau}^\ell(X; a, b)| \\
    &\cong \bigvee_{a, b}|j_!\overline{\Delta{\rm Cau}^\ell(X; a, b)}/j_!\overline{\Delta'{\rm Cau}^\ell(X; a, b)}|\\
   &\cong |\bigvee_{a, b}j_!\overline{\Delta{\rm Cau}^\ell(X; a, b)}/j_!\overline{\Delta'{\rm Cau}^\ell(X; a, b)}|.
    \end{align*}
    Now, by Proposition 4.2 of \cite{TY},  we have $\bigvee_{a, b}j_!\overline{\Delta{\rm Cau}^\ell(X; a, b)}/j_!\overline{\Delta'{\rm Cau}^\ell(X; a, b)} = {\sf M}^\ell_\bullet(X)$. 
\end{proof}

\end{document}